\documentclass{article}
\usepackage{amsmath,amsthm,amscd,amssymb}
\usepackage[mathscr]{eucal}
\usepackage{amssymb}
\usepackage{latexsym}

\theoremstyle{definition}
\newtheorem{Def}{Definition}[section]
\newtheorem{Thm}[Def]{Theorem}
\newtheorem{Prop}[Def]{Proposition}
\newtheorem{Rem}[Def]{Remark}
\newtheorem{Ex}[Def]{Example}

\numberwithin{equation}{section}

\title{On $p$-divisibility of Fourier coefficients of Siegel modular forms}

\author{Shoyu Nagaoka}

\date{}
\begin{document}

\maketitle

%\footnote{MSC 2020, Primary 11F33, Secondary 11F46.\\
%Key Words and phrases. Congruence for modular forms, Eisenstein series}

\begin{abstract}
\noindent
We describe the $p$-divisibility transposition for the Fourier coefficients
of Siegel modular forms. This provides a supplement to the result by Wilton for $p$-divisibility satisfied
by the Ramanujan $\tau$-function.
\end{abstract}

\section{Introduction}
\label{sec1}
The Fourier coefficients of modular forms are rich in number theoretic properties.
%%%
The motivation for the study lies in the congruence satisfied by the Ramanujan
 $\tau$-function $\tau(t)$. 
Function  $\tau(t)$ appears as the Fourier
coefficient of the elliptic cusp form $\Delta=q\prod_{n=1}^\infty(1-q^n)^{24}$
and it has many interesting properties (see, e.g., Serre \cite{Ser}).
One example of such properties
is the congruence relationship it satisfies.
It is known that the function $\tau(t)$ satisfies the congruence relation
$$
\tau (t) \equiv \sigma_{11}(t) \pmod{691},
$$
where $\sigma_{11}(t)=\sum_{0<d\mid t} d^{11}$. 

This is interpreted as the congruence relation between the modular form $\Delta$ and
the Eisenstein series $G_{12}$ of weight $12$:
$$
\Delta \equiv G_{12} \pmod{691},
$$
where $G_{12}$ is a constant multiple of the normalized Eisenstein
series $E_{12}$
(see $\S$ \ref{SE}, and Proposition \ref{Delta3}).
\\
As one of the results on $p$-divisibility,
we have the following classical result by Wilton \cite{W}:
\\
If $p$ is a prime number satisfying $\left(\frac{p}{23}\right)=-1$, then
$$
\tau(p) \equiv 0 \pmod{23},
$$
where $\left(\frac{\,p\,}{q}\right)$ is the Legendre symbol. This fact can also be paraphrased as
follows:
$$
\tau(t) \equiv 0 \pmod{23}\quad\text{if}\quad \chi_{-23}(t)=-1,
$$
where $\chi_{-23}(t)=\left(\frac{-23}{t}\right)$ is the Kronecker symbol.
\\
We consider the case of Siegel modular forms. It is known that the above modular form $\Delta$
can be lifted to a Siegel modular form $[\Delta]$ of degree 2. (To be precise, $[\Delta]$ is constructed
as the
Klingen Eisenstein series of degree 2 associated with $\Delta$ and satisfies $\Phi ([\Delta])=\Delta$
for the Siegel $\Phi$-operator by definition.)  B\"{o}cherer \cite{B1} proved that this modular form
is in the mod $23$ kernel of the theta operator:\;$\varTheta ([\Delta])\equiv 0 \pmod{23}$
 (see $\S$ \ref{ThetaOp}). This fact is consistent with its Fourier coefficient $a([\Delta],T)$ satisfying
$$
a([\Delta],T)\cdot\text{det}(T) \equiv 0 \pmod{23}.
$$
As indicated in Example \ref{Delta} of  $\S$ \ref{Eisen}, the Siegel modular form
$[\Delta]$ can be further lifted to a degree 3 Siegel modular form, which we denote by $[\Delta]^{(3)}$.
The constructed Siegel modular form $[\Delta]^{(3)}$ satisfies $\Phi([\Delta]^{(3)})=[\Delta]$,
and it has been proved that all of the Fourier coefficients $a([\Delta]^{(3)},T)$ for $T\in\Lambda_3^+$
 are divisible by $23$ (see Example \ref{Delta}).
These facts can be summarized as follows:
\vspace{2mm}
\\
(D-I)\; $a([\Delta]^{(3)},T) \equiv 0 \pmod{23}$ for all $T\in\Lambda_3^+$.
\vspace{1mm}
\\
(D-II)\; $a([\Delta],T) \equiv 0 \pmod{23}$ for $T\in\Lambda_2^+$ with $\text{det}(T)\not\equiv 0\pmod{23}$.
\vspace{1mm}
\\
(D-III)\; $a(\Delta,t)=\tau(t) \equiv 0 \pmod{23}$ for $t\in\mathbb{Z}_{>0}$ with $\chi_{-23}(t)=-1$.
\vspace{2mm}
\\
The above facts show that $p$-divisibility of the Fourier coefficients of a modular form
changes as its degree changes:
$$
\underset{\text{degree 3}}{[\Delta]^{(3)}}\;\;\overset{\Phi}{\longrightarrow}\;\;
\underset{\text{degree 2}}{[\Delta]}\;\;
\overset{\Phi}{\longrightarrow}\;\;\underset{\text{degree 1}}{\Delta}.
$$
We schematize the transposition of $p$-divisibility of the Fourier coefficients as follows.
\vspace{3mm}
\\
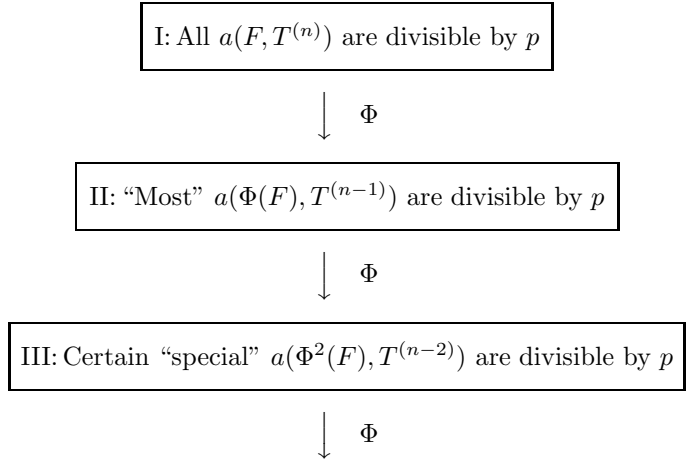
\begin{center}
\framebox(155,25)[c]{I:\,All $a(F,T^{(n)})$ are divisible by $p$}
\end{center}
$$
\Big{\downarrow}\quad \Phi
$$
\begin{center}
\framebox(205,25)[c]{II:\,``Most'' $a(\Phi (F),T^{(n-1)})$ are divisible by $p$}
\end{center}
$$
\Big{\downarrow}\quad \Phi
$$
\begin{center}
\framebox(255,25)[c]{III:\,Certain ``special'' $a(\Phi^2(F),T^{(n-2)})$ are divisible by $p$}
\end{center}
$$
\Big{\downarrow}\quad \Phi
$$
where $\Phi$ is the Siegel operator.
\\
\\
We will provide examples of modular forms that fit each of these stages by means of the
Siegel Eisenstein series.
\vspace{1mm}
\\
Let $E_k^{(n)}$ be the Siegel Eisenstein series for the Siegel modular group $\Gamma_n$ and
$a(E_k^{(n)},T)$ be the Fourier coefficients (see $\S$ \ref{SE}).
\vspace{2mm}
\\
The first result is as follows.
\vspace{2mm}
\\
%%%%%%%%%%%%%%%%%%%%%ThM-1
\textbf{Theorem M-1.}\quad {\it Let $p$ be a prime number satisfying $p>7$ and $p \equiv 3 \pmod{4}$.
Then we have
\vspace{2mm}
\\
{\rm (1-I)}\quad $a(E^{(3)}_{\frac{p+1}{2}},T) \equiv 0 \pmod{p}$ for all $T\in \Lambda_3^+$,
\vspace{2mm}
\\
{\rm (1-II)}\quad $a(E^{(2)}_{\frac{p+1}{2}},T) \equiv 0 \pmod{p}$ for $T\in \Lambda_2^+$ with
${\rm det}(T) \not\equiv 0 \pmod{p}$,
\vspace{2mm}
\\
{\rm (1-III)}\quad  $a(E^{(1)}_{\frac{p+1}{2}},t) \equiv 0 \pmod{p}$ for $t\in \mathbb{Z}_{>0}$ with
                   $\chi_{-p}(t)=-1$, where $\chi_{-p}(*)=\left(\frac{-p}{*}\right)$ is the Kronecker
                   symbol.}
\vspace{4mm}
\\
Statement (1-I) above means that the Siegel Eisenstein series $E^{(3)}_{\frac{p+1}{2}}$ is a mod $p$
singular modular form, and in the case of  (1-II), it is in the mod $p$ kernel of the theta operator
 (see $\S$ \ref{ThetaOp}).

The above theorem shows that the Siegel Eisenstein series $E_{12}^{(n)}$\,$(n=1,2,3)$ behaves
like $\Delta$ with respect to $p$-divisibility when considering the case $p=23$.
In $\S$ \ref{Eisen}, Example \ref{Leech}, we mention that the Siegel theta series $\vartheta_{\mathcal{L}}$
attached to the Leech lattice $\mathcal{L}$ also behaves similarly for $23$-divisibility.
It is interesting to note that there are several types of modular forms with similar properties at $p=23$.
\vspace{0mm}
\\
%%%%%%%%%%%%%%%%%% ThM-2
The second result is as follows.
\vspace{2mm}
\\
\textbf{Theorem M-2.}\quad {\it Let $p>5$ be a prime number. Then we have
\vspace{2mm}
\\
{\rm (2-I)}\quad $a(E_{p+1}^{(5)},T) \equiv 0 \pmod{p}$ for all $T\in\Lambda_5^+$,
\vspace{2mm}
\\
{\rm (2-II)}\quad $a(E_{p+1}^{(4)},T) \equiv 0 \pmod{p}$ for $T\in\Lambda_4^+$ with ${\rm det}(2T)\ne\square$,
\vspace{2mm}
\\
{\rm (2-II')}\quad $a(E_{p+1}^{(3)},T) \equiv 0 \pmod{p}$ for $T\in\Lambda_3^+$ with
                 ${\rm det}(T)\not\equiv 0 \pmod{p}$,
\vspace{2mm}
\\
{\rm (2-III)}\quad $a(E_{p+1}^{(2)},T) \equiv 0 \pmod{p}$ for $T\in\Lambda_2^+$ with
                 $\chi_T(p)=1$, where $\chi_T$ is the quadratic character defined by
                 $\chi_T(*)=\left(\frac{-{\rm det}(2T)}{*} \right)$.}
\vspace{3mm}
\\
%%%%%%%%%%%%%%%%%%%%%%%%%%%%%%%%%%%%%%%%%%%
For the general degree case, we have an example of modular forms that fit stages I and II above.
%%%
\vspace{2mm}
\\
\textbf{Theorem M-3.}\quad {\it
Assume that $n$ is even positive integer and $p$ is a prime number satisfying $p>n+3$ and
$p \equiv (-1)^\frac{n}{2} \pmod{4}$. We set $k=\frac{n+p-1}{2}$. Then there are modular
forms $F_k^{(n+1)}\in M_k(\Gamma_{n+1})_{\mathbb{Z}_{(p)}}$, $F_k^{(n)}\in M_k(\Gamma_n)_{\mathbb{Z}_{(p)}}$
satisfying $\Phi(F_k^{(n+1)})=F_k^{(n)}$ and
\vspace{2mm}
\\
{\rm (F-I)}\quad $a(F_k^{(n+1)},T) \equiv 0 \pmod{p}$ for all $T\in\Lambda_{n+1}^+$,
\vspace{1mm}
\\
{\rm (F-II)}\quad $a(F_k^{(n)},T) \equiv 0 \pmod{p}$ for $T\in\Lambda_n^+$ with ${\rm det}(T)\not\equiv 0 \pmod{p}$.
%%%%%%%%%%%%%% section 2%%%%%%%%%%%%%%%%%%%%%%%%%%%%%%%%%%%%%%%%%%%%
\section{Siegel modular forms}
%%%%%%%%% 2.1
\subsection{Siegel modular forms and their Fourier expansions}
\label{sec2.1}
{\rm
Let $M_k(\Gamma_n)$ be the space of modular forms of weight $k$ for the Siegel modular
group $\Gamma_n:=\text{Sp}_n(\mathbb{Z})$. Each $F\in M_k(\Gamma_n)$ admits a Fourier
expansion of the form
$$
F(Z)=\sum_{0\leq T\in\Lambda_n}a(F,T)q^T,\;\; q^T:=\text{exp}(2\pi\sqrt{-1}\text{tr}(TZ)),\;\;
Z\in\mathbb{H}_n,
$$
where $\mathbb{H}_n$ is the Siegel upper-half space of degree $n$ and
$$
\Lambda_n:=\{\,T=(t_{ij})\in \text{Sym}_n(\mathbb{Q})\,\mid \, t_{ii},\,2t_{ij} 
\in \mathbb{Z}\,\}.
$$
Denote by $\Lambda_n^+$ the subset consisting the positive elements of $\Lambda_n$.
\vspace{1mm}
\\
For a modular form $F\in M_k(\Gamma_n)$, the {\it Siegel} $\Phi$-{\it operator} is defined
by
$$
\Phi(F)(Z_1):=\lim_{\lambda\to+\infty}F\left(\begin{pmatrix}Z_1 & 0 \\ 0 & \sqrt{-1}\lambda \end{pmatrix}\right),
\quad Z_1\in\mathbb{H}_{n-1}.
$$
The operator $\Phi$ induces a linear map
$$
\Phi\,:\, M_k(\Gamma_n) \longrightarrow\,M_k(\Gamma_{n-1}).
$$
An element of the space $S_k(\Gamma_n):=\text{Ker}\,\Phi$ is called a {\it cusp form}.

For the Fourier coefficients, the following relationship holds:
$$
a(\Phi(F),T_1)=a\left(F,{\scriptsize \begin{pmatrix}T_1 & 0 \\ 0 & 0 \end{pmatrix}}\right)\quad
\text{for}\quad T_1\in\Lambda_{n-1}.
$$
For a subring $R\subset \mathbb{C}$, we set
$$
M_k(\Gamma_n)_R=\{\,F\in M_k(\Gamma_n)\,\mid \,a(F,T)\in R\;\text{for}\;\text{all}\;T\in\Lambda_n\;\}. 
$$
For a prime number $p$, we denote
by $\mathbb{Z}_{(p)}$ the local ring consisting of $p$-integral rational numbers.
}
%%%%%%%%%%% 2.2
\subsection{Theta operator}
\label{sec2.2}
{\rm
\label{ThetaOp}
For an element $F\in M_k(\Gamma_n)$, we define
$$
\varTheta\,:\, F=\sum a(F,T)q^T\, \longmapsto\, \varTheta(F):=\sum a(F,T)\cdot\text{det}(T)q^T
$$
and call it the {\it theta operator}. It should be noted that $\varTheta(F)$ is not necessarily a
Siegel modular form. For a modular form $F\in M_k(\Gamma_n)_{\mathbb{Z}_{(p)}}$, it happens that
$$
\varTheta(F) \equiv 0 \pmod{p}.
$$
In this case, we say that the modular form $F$ is an element of the {\it mod $p$ kernel
of the theta operator}. The Fourier coefficients $a(F,T)$ of such a form $F$ satisfies the following
condition:
$$
a(F,T) \equiv 0 \pmod{p}\;\text{for}\;T\in\Lambda_n^+\;\text{with}\;\text{det}(T)\not\equiv 0\pmod{p}.
$$
A modular form $F\in M_k(\Gamma_n)_{\mathbb{Z}_{(p)}}$ is called {\it mod $p$ singular}
if it satisfies
$$
a(F,T) \equiv 0 \pmod{p}\;\text{for}\;\text{all}\;T\in\Lambda_n^+.
$$
Of course, a mod $p$ singular modular form $F$ satisfies $\varTheta(F) \equiv 0 \pmod{p}$.
}
%%%%%%%%%%   2.3 Siegel Eisenstein
\subsection{Siegel Eisenstein series}
\label{SE}
{\rm
Let
$$
\Gamma_{n,\infty}:=\left\{ {\scriptsize \begin{pmatrix}A&B\\C&D \end{pmatrix}\in\Gamma_n\;\mid \;C=0_n }                        \; \right\}.
$$
For an even integer $k>n+1$, the {\it Siegel Eisenstein series of weight $k$} is
defined by
$$
E_k^{(n)}(Z):=\sum_{\binom{*\,*}{C\,D}\in\Gamma_{n,\infty}\backslash\Gamma_n}
\text{det}(CZ+D)^{-k},\quad Z\in\mathbb{H}_n.
$$
It is known that $E_k^{(n)}\in M_k(\Gamma_n)_{\mathbb{Q}}$ and
$$
\Phi(E_k^{(n)})=E_k^{(n-1)}.
$$
For $T\in\Lambda_n^+$\,($n$:\,even), we denote by
$\chi_T$ the primitive Dirichlet character corresponding the extension
$$
\mathbb{Q}\Big{(}\sqrt{ (-1)^{n/2}\text{det}(2T)}\,\Big{)}/\mathbb{Q},
$$
and further denote the conductor of $\chi_T$ by $f_T$.
\vspace{1mm}
\\
The following result is due to B\"{o}cherer.
%%
%%%%%%%%%%% Th 2.1
\begin{Thm} {\rm (B\"{o}cherer \cite{B2})}
\label{Bo1}
  {\it
{\rm (1)}\; If $n$ is even, then for $T\in\Lambda_n^+$,
$$
a(E_k^{(n)},T)=2^n\cdot\frac{k}{B_{k}}\cdot
\prod_{i=1}^{(n-2)/2}\frac{k-i}{B_{2k-2i}}\cdot\frac{B_{k-\frac{n}{2},\chi_T}}{B_{2k-n}}\cdot b_k^{(n)}(T)
$$
for some $b_k^{(n)}(T)\in\mathbb{Z}$.
\vspace{1mm}
\\
{\rm (2)}\; If $n$ is odd, then for  $T\in\Lambda_n^+$,
$$
a(E_k^{(n)},T)=2^n\cdot\frac{k}{B_k}\cdot\prod_{i=1}^{(n-1)/2}\frac{k-i}{B_{2k-2i}}\cdot c_k^{(n)}(T)
$$
for some $c_k^{(n)}(T)\in\mathbb{Z}$.
\vspace{1mm}
\\
Here $B_m$\,{\rm (}resp. $B_{m,\chi}${\rm )} is the $m$-th Bernoulli {\rm (}resp. 
generalized Bernoulli{\rm )} number.}
\end{Thm}
}
%%%%%%%%%%%%%%%%%%%%%%%%% 2.4 Ker Theta
\subsection{Mod $p$ kernel of theta operator}
\label{ThetaKer}
{\rm
We quote some results from \cite{N1} and \cite{N-T2} for examples of modular forms 
which are contained in the mod $p$ kernel of the theta operator.
\vspace{1mm}
\\
First, we assume that {\it $n$ is even}. We recall the formula for $a(E_k^{(n)},T)$
in Theorem \ref{Bo1}, (1) and set
$$
\alpha_p(n,k):=\text{ord}_p\left(\frac{k}{B_{k}}\cdot\prod_{i=1}^{(n-2)/2}\frac{k-i}{B_{2k-2i}}\right),
$$
where $\text{ord}_p(a)$ is the $p$-order of $a\in\mathbb{Q}$.
\\
The following theorem is a main result in \cite{N1}.
%%%%%%%%  Thm 2
\begin{Thm}
\label{Even}
{\rm (\cite[Theorem 3.1]{N1})}\quad {\it Assume that $n$ is an even positive integer. Let $p$ be a
prime number satisfying $p>n+3$ and $p \equiv (-1)^{n/2} \pmod{4}$. If we set
$$
G_k^{(n)}:=p^{-\alpha_p(n,k)}E_k^{(n)}\qquad \left(k=\frac{n+p-1}{2}\right),
$$
then $G_k^{(n)}\in M_k(\Gamma_n)_{\mathbb{Z}_{(p)}}$ and
$$
\varTheta(G_k^{(n)}) \equiv 0 \pmod{p}.
$$
}
\end{Thm}
Secondly, we assume that {\it $n$ is odd} and set
$$
\beta_p(n,k):=\text{ord}_p\left(\frac{k}{B_{k}}\cdot\prod_{i=1}^{(n-1)/2}\frac{k-i}{B_{2k-2i}}\right)=\alpha_p(n+1,k).
$$
In this case, we have the following partial result.
%%%
%%%%%%%%%%%%% Thm 3
\begin{Thm}
\label{Odd}
{\rm (\cite[Theorem 2.4]{N-T2})}\quad {\it Let $n$ be a positive integer such that $n \equiv 3 \pmod{8}$.
Assume that $p>n$ is a prime number. If we set
$$
H_k^{(n)}:=p^{-\beta_p(n,k)}E_k^{(n)}\qquad \left(k=\frac{n+2p-1}{2}\right),
$$
then $H_k^{(n)}\in M_k(\Gamma_n)_{\mathbb{Z}_{(p)}}$ and
$$
\varTheta(H_k^{(n)}) \equiv 0 \pmod{p}.
$$
}
\end{Thm}
%%%%%%%%%%
\begin{Rem}
It should be noted that the above theorem does not hold for all odd $n$.
\end{Rem}
}
%%%%%%%%%%%%%%%%%%%%%%%    3 p-div Siegel modular form%%%%%%%%%%%%%%%%%%%%%%%%%%%%%%%%%%%%%%%%%%%%%%%
\section{$p$-divisibility for Siegel Eisenstein series}
\label{Eisen}
%%%%%%%%%%%%%%%%%%%%%% 3.1 weight (p+1)/2
\subsection{Siegel Eisenstein series of weight $\frac{p+1}{2}$}
\label{Eisen1}
{\rm
In this section, we consider the Siegel Eisenstein series of weight $\frac{p+1}{2}$
in relation to the $p$-divisibility of the Fourier coefficients.

We have the following result.
%%%%%%%%%%%%%%%%%%%%%%%%%%%Theorem 3.1    Theorem 4
\begin{Thm}
\label{Th1}
{\it Let $p$ be a prime number satisfying $p>7$ and $p \equiv 3 \pmod{4}$.
Then we have
\vspace{2mm}
\\
{\rm (1-I)}\quad $a(E^{(3)}_{\frac{p+1}{2}},T) \equiv 0 \pmod{p}$ for all $T\in \Lambda_3^+$,
\vspace{2mm}
\\
{\rm (1-II)}\quad $a(E^{(2)}_{\frac{p+1}{2}},T) \equiv 0 \pmod{p}$ for $T\in \Lambda_2^+$ with
${\rm det}(T) \not\equiv 0 \pmod{p}$,
\vspace{2mm}
\\
{\rm (1-III)}\quad  $a(E^{(1)}_{\frac{p+1}{2}},t) \equiv 0 \pmod{p}$ for $t\in \mathbb{Z}_{>0}$ with
                   $\chi_{-p}(t)=-1$, where $\chi_{-p}(*)=\left(\frac{-p}{*}\right)$ is the Kronecker
                   symbol.}
\end{Thm}
%%%%%%% proof
\begin{proof}
(1-I):\;\; By Theorem \ref{Bo1}, we can write
$$
a(E^{(3)}_{\frac{p+1}{2}},T)=\frac{(p+1)\cdot  (p-1)}{B_{\frac{p+1}{2}}\cdot B_{p-1}}\cdot c_{1,p}(T)
$$
for some constant $c_{1,p}(T)\in\mathbb{Z}_{(p)}$. It is known that 
$B_{\frac{p+1}{2}}\not\equiv 0\pmod{p}$. In fact, $B_{\frac{p+1}{2}} \equiv -\frac{1}{2}h(-p) \pmod{p}$
holds for $p$ satisfying $p>3$ and $p \equiv 3 \pmod{4}$, where $h(-p)$ is the class number of
$\mathbb{Q}(\sqrt{-p})$
(see \cite{B-S}, Chap. 5,\, $\S$ 8, Problem 4).
On the other hand, from von Staudt-Clausen's theorem, $B_{p-1}^{-1} \equiv 0 \pmod{p}$.
Therefore,
$$
a(E^{(3)}_{\frac{p+1}{2}},T) \equiv 0 \pmod{p}\;\;\text{for}\;\;\text{all}\;\;T\in \Lambda_3^+.
$$
(1-II):\;\; We apply Theorem \ref{Even} to the case $n=2$ and $k=\frac{p+1}{2}$. Namely,
following the notation in $\S$ \ref{ThetaKer},
$$
G_{\frac{p+1}{2}}^{(2)}=p^{-\alpha_p(2,(p+1)/2)}E_{\frac{p+1}{2}}^{(2)}
$$
satisfies $\varTheta(G_{\frac{p+1}{2}}^{(2)}) \equiv 0 \pmod{p}$. 
Since $\alpha_p(2,(p+1)/2)=0$ in the present case, we obtain
$$
\varTheta(E^{(2)}_{\frac{p+1}{2}})=\varTheta(G_{\frac{p+1}{2}}^{(2)}) \equiv 0 \pmod{p}.
$$
This proves (1-II).
\vspace{1mm}
\\
(1-III):\;\; We can write
$$
a(E^{(1)}_{\frac{p+1}{2}},t)=-\frac{p+1}{B_{\frac{p+1}{2}}}\,\sigma_{\frac{p-1}{2}}(t),
$$
where $\sigma_{\frac{p-1}{2}}(t)=\sum_{0<d\mid t}d^{\frac{p-1}{2}}$. First we note that
$(p+1)/B_{\frac{p+1}{2}}\in\mathbb{Z}_{(p)}$. Next we shall show that
$$
\sigma_{\frac{p-1}{2}}(t) \equiv 0 \pmod{p}\;\;\text{if}\;\; \chi_{-p}(t)=-1.
$$
Assume that $\chi_{-p}(t)=-1$.
Let $\displaystyle t=\prod_{i=1}^mq_i^{e_i}$ be the prime decomposition of $t$. By assumption,
there is an integer $j$\,$(1\leq j\leq m)$ satisfying
$$
\chi_{-p}(q_j)=-1\quad\text{and}\quad e_j:\;\text{odd}.
$$
For this $j$,
$$
\sigma_{\frac{p-1}{2}}(t)=\prod_{i=1}^m\sum_{\ell=0}^{e_i}q_i^{\frac{p-1}{2}\ell}
\;\;\text{is}\;\text{divisible}\;\text{by}\;
1+q_j^{\frac{p-1}{2}}.
$$
(Since $e_j$ is odd, $\sum_{\ell=0}^{e_j}q_j^{\frac{p-1}{2}\ell}$ is divisible by $1+q_j^{\frac{p-1}{2}}$.)
By Euler's criterion,
$$
1+q_j^{\frac{p-1}{2}} \equiv 1+\left(\frac{q_j}{p}\right)=1+\chi_{-p}(q_j)=0 \pmod{p}.
$$
This implies that
$$
\sigma_{\frac{p-1}{2}}(t) \equiv 0 \pmod{p}\;\;\text{if}\;\;\chi_{-p}(t)=-1.
$$
This completes the proof of Theorem \ref{Th1}.
\end{proof}
}
%%%%%%%%%%%%% proof end
%%%%%%% Remark3.2    Remark 2
\begin{Rem}
In the higher degree case $(n>3)$, we have the following result under some additional conditions
on $p$: Assume that $p>2n+1$ and $p \equiv 3 \pmod{4}$, and $p$ is {\it regular}. Then
$$
a(E_{\frac{p+1}{2}}^{(n)},T) \equiv 0 \pmod{p}\;\;\text{for}\;\;\text{all}\;\;T\in\Lambda_n^+.
$$
\end{Rem}
%%%%%%%%%%%%%%%%%%%%%%%%%%%%%%%%%% Remark 3.3       Remark 3
\begin{Rem}
\label{Remark3}
Let $p=23$ in the above theorem. We have
\vspace{1mm}
\\
(E-I)\quad $a(E^{(3)}_{12},T) \equiv 0 \pmod{23}$ for all $T\in \Lambda_3^+$,
\vspace{2mm}
\\
(E-II)\quad $a(E^{(2)}_{12},T) \equiv 0 \pmod{23}$ for $T\in \Lambda_2^+$ with
${\rm det}(T) \not\equiv 0 \pmod{23}$,
\vspace{2mm}
\\
(E-III)\quad  $a(E^{(1)}_{12},t) \equiv 0 \pmod{23}$ for $t\in \mathbb{Z}_{>0}$ with
                   $\chi_{-23}(t)=-1$.
\vspace{2mm}
\\
In the Introduction, it was noted that there are modular forms $\Delta$, $[\Delta]$, and $[\Delta]^{(3)}$
which behave similarly with respect to ``23-divisibility''. Details are given below.
\end{Rem}
%%%%%  Example 3.4  Delta    Example 1
\begin{Ex}\textbf{(Modular forms related to $\boldsymbol{\Delta}$)}
\label{Delta}
\quad In the Introduction, we mentioned the weight 12 modular forms $\Delta$, $[\Delta]$, and
$[\Delta]^{(3)}$ which have the following properties:
\vspace{2mm}
\\
(D-I)\; $a([\Delta]^{(3)},T) \equiv 0 \pmod{23}$ for all $T\in\Lambda_3^+$,
\vspace{2mm}
\\
(D-II)\; $a([\Delta],T) \equiv 0 \pmod{23}$ for $T\in\Lambda_2^+$ with $\text{det}(T)\not\equiv 0\pmod{23}$,
\vspace{2mm}
\\
(D-III)\; $a(\Delta,t)=\tau(t) \equiv 0 \pmod{23}$ for $t\in\mathbb{Z}_{>0}$ with $\chi_{-23}(t)=-1$.
\vspace{2mm}
\\
First, we review the properties of $\Delta$.
\vspace{1mm}
\\
The modular form $\Delta$ is defined by
$$
\Delta=q\prod_{n=1}^\infty(1-q^n)^{24},\quad q=e^{2\pi\sqrt{-1}z}.
$$
and has the Fourier expansion
$$
\Delta=\sum_{t=1}^\infty \tau(t)q^t\in S_{12}(\Gamma_1)_{\mathbb{Z}}.
$$
The function $\tau (t)$ is called the {\it Ramanujan $\tau$-function}. As mentioned in
the Introduction, $\tau(t)$ satisfies
$$
\tau(t) \equiv 0 \pmod{23}\;\;\text{if}\;\; \chi_{-23}(t)=-1\qquad (\text{Wilton \cite{W}}).
$$
This is none other than (D-III).
\vspace{2mm}
\\
In the case that $n=2$, we consider the Klingen Eisenstein series $[\Delta]$ lifted
from $\Delta$. From the construction method, we see that $\Phi([\Delta])=\Delta$ and 
$[\Delta]\in M_{12}(\Gamma_2)$ ($\Phi$:\,the Siegel $\Phi$-operator). Moreover, the following
fact was proved by B\"{o}cherer \cite{B1}:
$$
\varTheta([\Delta]) \equiv 0 \pmod{23}.
$$
This implies (D-II).
\vspace{2mm}
\\
Finally, we construct $[\Delta]^{(3)}$ concretely.
\vspace{1mm}
\\
We set
\begin{align*}
[\Delta]^{(3)}:=&c_1\cdot (E_4^{(3)})^3+c_2\cdot (E_6^{(3)})^2+c_3\cdot E_{12}^{(3)}+c_4\cdot F_{12},\\
                   & c_1=\frac{191\cdot 691}{2^4\cdot 3^3\cdot 5^3\cdot 7\cdot 337},\qquad
                      c_2=\frac{-277\cdot 691}{2^6\cdot 3^4\cdot 7^3\cdot 337}\\
                   & c_3=\frac{-131\cdot 593\cdot 691}{2^6\cdot 3^4\cdot 5^3\cdot 7^3\cdot 337}, \qquad
                      c_4=\frac{2^2\cdot 5^3\cdot 71\cdot 691}{7\cdot 337}.
\end{align*}
Here $E_k^{(3)}$ is the weight $k$ Siegel Eisenstein series of degree 3, and $F_{12}\in S_{12}(\Gamma_3)$
is a cusp form constructed by Miyawaki \cite{Mi}.

Function $[\Delta]^{(3)}$ has the following properties:
%%%%%%%%%%%%%%%%%%%%%%%%%%%%%%5 Prop. 3.5      Prop. 5
\begin{Prop}
\label{Delta3}
{\it 
{\rm (1)} \;$[\Delta]^{(3)}\in M_{12}(\Gamma_3)_{\frac{1}{7}\mathbb{Z}}$.
\qquad\qquad
{\rm (2)}\; $\Phi([\Delta]^{(3)})=[\Delta]$.
\vspace{1mm}
\\
{\rm (3)}\; $a([\Delta]^{(3)},T) \equiv 0 \pmod{23}$ for all $T\in \Lambda_3^+$.
\vspace{1mm}
\\
{\rm (4)}\; $[\Delta]^{(3)} \equiv \mathcal{G}_{12}^{(3)} \pmod{691}$ where
$\mathcal{G}_{12}^{(3)}:=-\frac{B_{12}}{24}\cdot E_{12}^{(3)}$.
}
\end{Prop}
%%%%%%%%%%% prrof Prop 3.5
\begin{proof}
(1)\; Fact (1) follows
from the expression 
$[\Delta]^{(3)}=Y_{12}^{(3)}-\frac{2^4\cdot 3^2\cdot 5}{7}X_{12}^{(3)}+3\cdot 11\cdot 17\cdot 23\,F_{12}$, 
where $Y_{12}^{(3)}$ and $X_{12}^{(3)}$ are modular forms in $M_{12}(\Gamma_3)_{\mathbb{Z}}$ defined in
\cite[$\S$ 4.1]{N-T1}, and $F_{12}\in S_{12}(\Gamma_3)_{\mathbb{Z}}$.
\\
(2)\; Equality (2) is obtained from the definition of
$[\Delta]^{(3)}$ and the fact that $\Phi(E_k^{(3)})=E_k^{(2)}$.
\\
(3)\; As will be stated in Example \ref{Leech}, the congruence relation
$E_{12}^{(3)}+7\cdot [\Delta]^{(3)} \equiv \vartheta_{\mathcal{L}}^{(3)} \pmod{23}$ holds,
where $\vartheta_{\mathcal{L}}^{(3)}$ is the Siegel theta series associated with the
Leech lattice $\mathcal{L}$ (see Example \ref{Leech}).
Since $a(E_{12}^{(3)},T) \equiv a(\vartheta_{\mathcal{L}}^{(3)},T) \equiv 0 \pmod{23}$
for all $T\in\Lambda_3^+$ (see, Remark \ref{Remark3}, (E-I) and Example \ref{Leech}, (L-I)), 
we can confirm statement (3).
\\
(4)\; For $\mathcal{G}_{12}^{(3)}=-\frac{B_{12}}{24}\cdot E_{12}^{(3)}$, we obtain 
\begin{align*}
[\Delta]^{(3)}-\mathcal{G}_{12}^{(3)}=&c_1\cdot (E_4^{(3)})^3+c_2\cdot (E_6^{(3)})^2+c'_3\cdot E_{12}^{(3)}
+c_4\cdot F_{12},\\
         c'_3=&-\frac{103\cdot 223\cdot 691^2}{2^6\cdot 3^4\cdot 5^3\cdot 7^3\cdot 13\cdot 337}.
\end{align*}
The respective Fourier coefficients of $c_1\cdot (E_4^{(3)})^3$, $c_2\cdot (E_6^{(3)})^2$, and $c_4\cdot F_{12}$ can
be divisible by $691$. Since
$c'_3\cdot E_{12}^{(3)}$ is divisible by $691$ (note that $691^2\mid c'_3$), we obtain
$[\Delta]^{(3)} \equiv \mathcal{G}_{12}^{(3)} \pmod{691}$.
\end{proof}
%%%%%%%%%%%%%%%%
Statement (3) above shows that claim (D-I) is established, and statement (4)
is a generalization of ``mod $691$ congruence'' between $\Delta$ and 
$G_{12}=\mathcal{G}_{12}^{(1)}=-\frac{B_{12}}{24}\cdot E_{12}^{(1)}$ stated in the Introduction.
\vspace{2mm}
\\
\noindent
%%%%%%%%%% Numerical Ex
{\it Numerical examples}
\vspace{1mm}
\\
We next give some numerical examples showing the validity of claim (3) above:
\begin{align*}
& a([\Delta]^{(3)},[1,1,1;0,0,0])=\frac{2^3\cdot\underline{23}\cdot 487}{7},\quad
a([\Delta]^{(3)},[1,1,1;0,0,1])=2^3\cdot 3\cdot\underline{23}\cdot 419,\\
&  a([\Delta]^{(3)},[1,1,1;1,1,1])=\frac{2^3\cdot \underline{23}\cdot 80429}{7} .  
\end{align*}
Above, we used the abbreviation
$$
[a,b,c,;d,e,f]:={\scriptsize
\begin{pmatrix}a & f/2 & e/2 \\ f/2 & b  & d/2 \\ e/2 & d/2 & c
\end{pmatrix}}\in\Lambda_3.
$$
\end{Ex}
%%%%%%%% Example 3.6   Leech
\begin{Ex}
\label{Leech}
\textbf{(Siegel theta series for the Leech lattice)}\quad Let $\mathcal{L}$ be the Leech lattice, 
which is one of the
Niemeier lattices of rank 24. We consider the Siegel theta series associated with $\mathcal{L}$:
$$
\vartheta_{\mathcal{L}}^{(n)}(Z)=
\sum_{X\in M_{24,n}(\mathbb{Z})}\text{exp}(\pi\sqrt{-1}\,\text{tr}(L[X]Z)),\quad Z\in\mathbb{H}_n,
$$
where $L$ is the Gram matrix for $\mathcal{L}$ and $L[X]:={}^tXLX$. The series represents a
Siegel modular form of weight 12:
$$
\vartheta_{\mathcal{L}}^{(n)}\in M_{12}(\Gamma_n)_{\mathbb{Z}}.
$$
Moreover, we have
\vspace{2mm}
\\
(L-I)\; $a(\vartheta_{\mathcal{L}}^{(3)},T) \equiv 0 \pmod{23}$ for all $T\in\Lambda_3^+$,
\vspace{2mm}
\\
(L-II)\; $a(\vartheta_{\mathcal{L}}^{(2)},T) \equiv 0 \pmod{23}$ for $T\in\Lambda_2^+$ with 
$\text{det}(T)\not\equiv 0\pmod{23}$,
\vspace{2mm}
\\
(L-III)\; $a(\vartheta_{\mathcal{L}}^{(1)},t) \equiv 0 \pmod{23}$ for $t\in\mathbb{Z}_{>0}$ with $\chi_{-23}(t)=-1$.
\vspace{2mm}
\\
Statements (L-I) and (L-II) are results in \cite[Theorem 8, (2)]{N-T1} and \cite[Theorem 8, (1)]{N-T1}
respectively.
\vspace{1mm}
\\
Among modular forms $E_{12}^{(3)}$, $[\Delta]^{(3)}$, and $\vartheta_{\mathcal{L}}^{(3)}$,
the following congruence relation holds:
\begin{equation}
\label{mod23}
E_{12}^{(3)}+7\cdot [\Delta]^{(3)} \equiv \vartheta_{\mathcal{L}}^{(3)} \pmod{23}.
\end{equation}
This congruence relation is obtained from the concrete representation of these
modular forms by $(E_4^{(3)})^3$,\,$Y_{12}^{(3)}$,\,$X_{12}^{(3)}$, and Miyawaki's cusp form $F_{12}$:
\begin{align*}
E_{12}^{(3)}&=d_1\cdot (E_4^{(3)})^3+d_2\cdot Y_{12}^{(3)}+d_3\cdot X_{12}^{(3)}+d_4\cdot F_{12},\\
& d_1=1,\quad d_2=-\tfrac{2^7\cdot 3^3\cdot 5^3}{691},\quad  
d_3=\tfrac{2^{15}\cdot 3^5\cdot 5^3\cdot 1759}{131\cdot 593\cdot 691},\quad
d_4=-\tfrac{2^7\cdot 3^3\cdot 5^3\cdot 383\cdot 363343}{131\cdot 593\cdot 691}.\\
[\varDelta]^{(3)}&=e_1\cdot (E_4^{(3)})^3+e_2\cdot Y_{12}^{(3)}+e_3\cdot X_{12}^{(3)}+e_4\cdot F_{12},\\
& e_1=0,\quad e_2=1,\quad e_3=-\tfrac{2^4\cdot 3^2\cdot 5}{7},\quad
e_4=3\cdot 11\cdot 17\cdot 23.\\
\vartheta_{\mathcal{L}}^{(3)}&=f_1\cdot (E_4^{(3)})^3+f_2\cdot Y_{12}^{(3)}+f_3\cdot X_{12}^{(3)}+f_4\cdot F_{12},\\
& f_1=1,\quad f_2=-2^4\cdot3^2\cdot 5,\quad f_3=2^6\cdot 3^3\cdot 5^2,\quad
f_4=-2^4\cdot 3^2\cdot 5\cdot 1571.
\end{align*}
The required congruence (\ref{mod23}) is derived from
$$
d_i+7e_i \equiv f_i \pmod{23}\quad (i=1,2,3,4).
$$
This congruence means that the ``mod 23 properties'' of $\vartheta_{\mathcal{L}}^{(3)}$
are induced by those of $E_{12}^{(3)}$ and $[\Delta]^{(3)}$.  In particular, fact (L-III) is a consequence
of this congruence.
\end{Ex}
%%%%%%%%% %%%%%%%%%%%%%%%%%%%% %%%%%%%%%%%%3.2   weight p+1
\subsection{Siegel Eisenstein series of weight $p+1$}
\label{Eisen2}
We consider the Siegel Eisenstein series of weight $p+1$.
%%%%%%%%%%%%%%%%%%%%%%%% Th 3.7
\begin{Thm}
\label{Th2}
{\it Let $p>5$ be a prime number.
Then we have
\vspace{2mm}
\\
{\rm (2-I)}\quad  $a(E^{(5)}_{p+1},T) \equiv 0 \pmod{p}$ for all $T\in \Lambda_5^+$,
\vspace{2mm}
\\
{\rm (2-II)}\quad $a(E^{(4)}_{p+1},T) \equiv 0 \pmod{p}$ for $T\in \Lambda_4^+$ with ${\rm det}(2T)\ne\square$,
\vspace{2mm}
\\
{\rm (2-II')}\quad $a(E^{(3)}_{p+1},T) \equiv 0 \pmod{p}$ for $T\in \Lambda_3^+$ with
${\rm det}(T) \not\equiv 0 \pmod{p}$,
\vspace{2mm}
\\
{\rm (2-III)}\quad   $a(E_{p+1}^{(2)},T) \equiv 0 \pmod{p}$ for $T\in\Lambda_2^+$ with
                 $\chi_T(p)=1$, 
where $\chi_T$ is the quadratic character given in $\S$ \ref{SE}.}
\end{Thm}
%%%%%%%%%%%%%%%%%%%%%    proof 3.7
\begin{proof}
%%%%%%%%%  (2-I)
(2-I):\; By Theorem \ref{Bo1}, (2), for $T\in\Lambda_5^+$, we can write
$$
a(E_{p+1}^{(5)},T)=\frac{p+1}{B_{p+1}}\cdot\frac{p}{B_{2p}}
\cdot\frac{p-1}{B_{2(p-1)}}\cdot d_{1,p}(T)
$$
with $d_{1,p}\in\mathbb{Z}_{(p)}$. We examine the $p$-divisibility of each factor on the
right-hand side of the above identity.
\\
By Kummer's congruence relation, $B_{p+1}/(p+1)\equiv B_2/2=1/12 \pmod{p}$.
Hence, $(p+1)/B_{p+1}\in\mathbb{Z}_{(p)}$.
By a theorem of Adams (see, e.g., \cite[Chap.5, \textsc{Exercises} 5.10]{Wa}),
we see that $B_{2p} \equiv 0 \pmod{p}$. However, it follows from Kummer's congruence
that $B_{2p}/(2p) \equiv B_{p+1}/(p+1) \equiv 1/12 \pmod{p}$. Hence, $p/B_{2p}\in\mathbb{Z}_{(p)}$.
Regarding the factor $(p-1)/B_{2(p-1)}$, we have  $(p-1)/B_{2(p-1)}\in p\mathbb{Z}_{(p)}$ by
von Staudt-Clausen's theorem. Hence, we see that
$a(E_{p+1}^{(5)},T) \equiv 0 \pmod{p}$ for all $T\in\Lambda_5^+$.
\\
%%%%%%%%%  (2-II)
(2-II):\; Assume that $T\in\Lambda_4^+$. Then, by Theorem \ref{Bo1}, (1),
we can write
$$
a(E_{p+1}^{(4)},T)=\frac{p+1}{B_{p+1}}\cdot\frac{p}{B_{2p}}
\cdot\frac{B_{p-1,\chi_T}}{B_{2(p-1)}}\cdot d_{2,p}(T)
$$
for some constant $d_{2,p}(T)\in\mathbb{Z}_{(p)}$.
As in the case of (2-I), we see that the factor $((p+1)/B_{p+1})\cdot(p/B_{2p})$ is $p$-integral.
\\
We examine the factor $B_{p-1,\chi_T}/B_{2(p-1)}$. We note
that $B_{2(p-1)}^{-1}\in p\mathbb{Z}_{(p)}$ is derived from von Staudt-Clausen's theorem.
Regarding the factor $B_{p-1,\chi_T}$, it happens that the character $\chi_T$ is trivial in this case.
In fact, if $\text{det}(2T)$ is square, then $\chi_T$ becomes the trivial character, and we obtain
$$
\frac{B_{p-1,\chi_T}}{B_{2(p-1)}}=\frac{B_{p-1}}{B_{2(p-1)}} \equiv 1 \pmod{p}
$$ 
by von Staudt-Clausen's theorem. If $\text{det}(2T)$ is not square, then the congruence
$$
\frac{B_{p-1,\chi_T}}{B_{2(p-1)}} \equiv 0 \pmod{p}
$$
holds because the denominator of $B_{m,\chi_T}$ is divisible by $p$ only when
$f_T=p$ and $m=\frac{p-1}{2}\cdot t$ for an odd integer $t$\,
(see Carlitz \cite[Theorem 3]{C}). Consequently, the factor $B_{p-1,\chi_T}/B_{2(p-1)}$,
and therefore, $a(E_{p+1}^{(4)},T)$ is divisible by $p$ if $\text{det}(2T)$ is not square.
\\
%%%%%%%%   (2-II')
(2-II'):\; It is sufficient to show that $\varTheta (E_{p+1}^{(3)}) \equiv 0 \pmod{p}$.
By Theorem \ref{Odd}, if we set
$$
\beta_p(3,p+1)=\text{ord}_p(((p+1)\cdot p)/(B_{p+1}\cdot B_{2p})),
$$
then $H_{p+1}^{(3)}=p^{-\beta_p(3,p+1)}\cdot E_{p+1}^{(3)}$ satisfies
$\varTheta(H_{p+1}^{(3)}) \equiv 0 \pmod{p}$.
Since $\beta_p(3,p+1)=0$ in the present case (see (2-I)), we obtain
$$
\varTheta (E_{p+1}^{(3)})=\varTheta(H_{p+1}^{(3)}) \equiv 0 \pmod{p}.
$$
%%%%%%%%%%% (2-III)
(2-III):\; For $T\in\Lambda_2^+$, we can write
$$
a(E_{p+1}^{(2)},T)=\frac{p+1}{B_{p+1}}\cdot\frac{p}{B_{2p}}\cdot\frac{B_{p,\chi_T}}{p}\cdot d_{3,p}(T)
$$
for some $d_{3,p}(T)\in\mathbb{Z}_{(p)}$. We examine each factor on the right-hand side like in 
the case of (2-I) and (2-II). 
For $p$-integrality of the factors $(p+1)/B_{p+1}$ and $p/B_{2p}$, we have already shown it in 
the case of (2-I). 

We shall show that the factor $B_{p,\chi_T}/p$ is divisible by $p$ when $\chi_T(p)=1$.\\
In the case that $p\nmid f_T$, the following congruence relation holds:
$$
\frac{B_{p,\chi_T}}{p} \equiv \frac{1}{f_T}(1-\chi_T(p))\sum_{s=1}^{f_T} s\chi_T(s) \pmod{p}
$$
(see Carlitz \cite[p.182, (5.5)]{C}).
\\
When $p\mid f_T$, we see that
$\displaystyle B_{p,\chi_T}/p\in\mathbb{Z}_{(p)}$
(see Carlitz \cite[Theorem 3]{C}).\\
Combining these facts, we obtain
$$
\frac{B_{p,\chi_T}}{p} \equiv 0 \pmod{p}\;\;\text{if}\;\; \chi_T(p)=1.
$$
Consequently, we see that
$$
a(E_{p+1}^{(2)},T) \equiv 0 \pmod{p}\;\;\text{if}\;\;\chi_T(p)=1.
$$
This completes the proof of Theorem \ref{Th2}.
\end{proof}
%%%%%%%%%%% proof end
%%%%%%%%%%%%%%%%%%%%%%%%%%%%%  3.3 General n
\subsection{General degree case}
{\rm For general degree $n$, we have an example of modular forms that fit stages I and II.}
%%%%%%%%%%%   Th 3.8
\begin{Thm}
\label{Gen}
\;{\it
Assume that $n$ is an even positive integer and $p$ is a prime number satisfying $p>n+3$ and
$p \equiv (-1)^\frac{n}{2} \pmod{4}$. We set $k=\frac{n+p-1}{2}$. Then there are modular
forms $F_k^{(n+1)}\in M_k(\Gamma_{n+1})_{\mathbb{Z}_{(p)}}$ and $F_k^{(n)}\in M_k(\Gamma_n)_{\mathbb{Z}_{(p)}}$
satisfying $\Phi(F_k^{(n+1)})=F_k^{(n)}$ and
\vspace{2mm}
\\
{\rm (F-I)}\quad $a(F_k^{(n+1)},T) \equiv 0 \pmod{p}$ for all $T\in\Lambda_{n+1}^+$,
\vspace{2mm}
\\
{\rm (F-II)}\quad $a(F_k^{(n)},T) \equiv 0 \pmod{p}$ for $T\in\Lambda_n^+$ with ${\rm det}(T)\not\equiv 0 \pmod{p}$.}
\end{Thm}
%%%%%%
\begin{proof}
We apply Theorem \ref{Even}. As $F_k^{(n)}$, we adopt $G_k^{(n)}$ in Theorem \ref{Even}.
Namely,
\begin{align*}
F_k^{(n)}&:=G_k^{(n)}=p^{-\alpha_p(n,k)}\cdot E_k^{(n)},\\
            & \alpha_p(n,k)=\text{ord}_p\left(\frac{k}{B_{k}}\cdot\prod_{i=1}^{(n-2)/2}\frac{k-i}{B_{2k-2i}}\right),
\quad k=\frac{n+p-1}{2}.
\end{align*}
Then, by Theorem \ref{Even}, it follows that
$$
\varTheta (F_k^{(n)})=\varTheta(G_k^{(n)}) \equiv 0 \pmod{p}.
$$
This proves statement (F-II). For $F_k^{(n+1)}$, we use
$$
F_k^{(n+1)}:=p^{-\alpha_p(n,k)}\cdot E_k^{(n+1)},
$$
where we adopt $p^{-\alpha_p(n,k)}$ as the multiplying constant. Therefore, we see that
$\Phi(F_k^{(n+1)})=F_k^{(n)}$.

By Theorem \ref{Bo1}, (2), we can write
$$
a(E_k^{(n+1)},T)=\frac{n+p-1}{B_{\frac{n+p-1}{2}}}\cdot
                     \frac{(n+p-3)\cdot\cdots\cdot (p+1)}{B_{n+p-3}\cdot\cdots\cdot B_{p+1}}\cdot
                     \frac{p-1}{B_{p-1}}\cdot d_p(T)
$$
for some $d_p(T)\in\mathbb{Z}_{(p)}$. Since
$$
\alpha_p(n,k)=\text{ord}_p\left( \frac{n+p-1}{B_{\frac{n+p-1}{2}}}\cdot
                     \frac{(n+p-3)\cdot\cdots\cdot (p+1)}{B_{n+p-3}\cdot\cdots\cdot B_{p+1}} \right),
$$
we obtain
\begin{align*}
\text{ord}_p(a(F_k^{(n+1)},T)) &=\text{ord}_p(a(E_k^{(n+1)},T))-\alpha_p(n,k)\\
                                         & \geq \text{ord}_p((p-1)/B_{p-1})=1,
\end{align*}
by von Staudt-Clausen's theorem.  Therefore,
$$
a(F_k^{(n+1)},T) \equiv 0 \pmod{p}\;\;\text{for}\;\;\text{all}\;\;T\in\Lambda_{n+1}^+.
$$
This proves (F-I) and completes the proof of Theorem \ref{Gen}.
\end{proof}
%%%%%%%%%%%%%%%%%%%%%%%%%%%%%%%%%%%%%%%%%%%%%%%%%%%%%%%%%%%%%%%%%%%%%%%%%%%%%%%%%%%%
\section{On the relationship with $p$-adic Siegel Eisenstein series}
\label{sec4}
{\rm
In this section, we mention the relationship between some of the
results obtained above and the $p$-adic Eisenstein series.
\vspace{2mm}
\\
First, we consider the case of the Eisenstein series of weight $\frac{p+1}{2}$ treated in
$\S$ \ref{Eisen1}.
\\
We assume that $p>3$ and $p \equiv 3 \pmod{4}$.  In \cite{N0}, the following identity was proved:
$$
\lim_{m\to\infty}E_{1+\frac{1}{2}p^{m-1}(p-1)}^{(n)}=\text{genus}\,\vartheta^{(n)}(B_{-p}^{(2)}).
$$
Here, the left-hand side is the $p$-adic limit of the sequence of Eisenstein series 
$$
\left\{E_{1+\frac{1}{2}p^{m-1}(p-1)}^{(n)}\right\}_{m=1}^\infty
$$
and the right-hand side is the genus theta series associated with binary quadratic forms of discriminant
$-p$ and level $p$. Considering the first approximation of the limit (namely, $m=1$), we obtain
$$
E_{\frac{p+1}{2}}^{(3)} \equiv \lim_{m\to\infty}E_{1+\frac{1}{2}p^{m-1}(p-1)}^{(3)}
                              =\text{genus}\,\vartheta^{(3)}(B_{-p}^{(2)}) \pmod{p}.
$$
Statements (1-I) and (1-II) in Theorem \ref{Th1} can be also derived from this equality.

Next we consider the case of weight $p+1$.
\vspace{1mm}
\\
In \cite{K-N}, the following identity was proved:
$$
\lim_{m\to\infty}E_{2+p^{m-1}(p-1)}^{(n)}=\text{genus}\,\vartheta^{(n)}(S_{p^2}^{(4)}).
$$
Here the left-hand side is the $p$-adic Eisenstein series in the above case, and the right-hand
side is the genus theta series associated with quaternary quadratic forms of discriminant $p^2$
and level $p$. As in the above case, by considering the first approximation of the limit, we
obtain
$$
E_{p+1}^{(5)} \equiv \lim_{m\to\infty}E_{2+p^{m-1}(p-1)}^{(5)}=\text{genus}\,\vartheta^{(5)}(S_{p^2}^{(4)})
\pmod{p}.
$$
Statements (2-I) and (2-II) in Theorem \ref{Th2} are also proved using this identity.
}
%%%%%%%%%%%%
\section{Remark}
{\rm 
In $\S$ 3, we proved
$$
[\Delta]^{(3)} \equiv G_{12}^{(3)} \pmod{691}\quad \text{and}\quad
E_{12}^{(3)}+7\cdot [\Delta]^{(3)} \equiv \vartheta_{\mathcal{L}}^{(3)} \pmod{23},
$$
(see, Proposition \ref{Delta3}, (3), and Example \ref{Leech}).
These results lead to the following problem.
\vspace{2mm}
\\
\textbf{Problem:}\quad Construct a modular form $[\Delta]^{(n)}$ in $M_{12}(\Gamma_n)$
that satisfies all of the following conditions:
\begin{align*}
& (1)\quad \Phi^{n-1}([\varDelta]^{(n)})=\varDelta,\qquad
(2)\quad E_{12}^{(n)}+7\cdot [\varDelta]^{(n)} \equiv \vartheta_L^{(n)} \pmod{23}\\
& (3)\quad[\varDelta]^{(n)} \equiv \mathcal{G}_{12}^{(n)} \pmod{691}.
\end{align*}
Here $\mathcal{G}_{12}^{(n)}:=-\frac{B_{12}}{24}\cdot E_{12}^{(n)}$.
}
%%%%%%%%%%%%%%%%%%%%%%%%%%%%%%%%%%%%%%%%%%%
{\rm 

}

\end{document}